\documentclass[10pt]{amsart}
\usepackage{graphicx}
\usepackage{amsmath,amssymb,xy,array}
\usepackage[T1]{fontenc}
\usepackage{amsfonts}
\usepackage{mathabx}
\usepackage{amsmath}
\usepackage{amssymb}
\usepackage{amsthm}
\usepackage{graphicx}
\usepackage{amscd}
\usepackage{latexsym}
\usepackage{enumitem}
\usepackage{xcolor}
\usepackage{amsfonts}
\usepackage{xypic}
\usepackage[utf8]{inputenc}
\usepackage{hyperref}
\usepackage{mathrsfs}
\usepackage[toc,page]{appendix}
\usepackage{fancyhdr}
\usepackage{tikz}
\usepackage{tikz-cd}
\usepackage{longtable}
\usepackage{geometry}
\usepackage{textcomp}
\usetikzlibrary{trees}
\usetikzlibrary{arrows}
\usepackage{multirow}
\usepackage{youngtab}
\usepackage{mathdots}

\definecolor{yqyqyq}{rgb}{0.5019607843137255,0.5019607843137255,0.5019607843137255}\definecolor{uuuuuu}{rgb}{0.26666666666666666,0.26666666666666666,0.26666666666666666}
\definecolor{uququq}{rgb}{0.25098039215686274,0.25098039215686274,0.25098039215686274}
\definecolor{wwwwww}{rgb}{0.4,0.4,0.4}
\definecolor{uuuuuu}{rgb}{0.26666666666666666,0.26666666666666666,0.26666666666666666}

\setlist[itemize]{leftmargin=6mm}

\renewcommand{\P}{\mathbb P}

\newcommand{\ZZ}{\mathbb Z}
\newcommand{\GG}{\mathbb G}
\DeclareMathOperator{\codim}{codim}

\DeclareMathOperator{\Cl}{Cl}

\DeclareMathOperator{\NE}{NE}

\DeclareMathOperator{\lin}{lin}

\DeclareMathOperator{\mult}{mult}

\DeclareMathOperator{\Hom}{Hom}

\DeclareMathOperator{\Exc}{Exc}

\DeclareMathOperator{\Sing}{Sing}

\DeclareMathOperator{\MCD}{MCD}
\DeclareMathOperator{\SBLD}{SBLD}

\DeclareMathOperator{\Eff}{Eff}
\DeclareMathOperator{\Nef}{Nef}
\DeclareMathOperator{\Mov}{Mov}
\DeclareMathOperator{\Pic}{Pic}
\DeclareMathOperator{\rank}{rank}

\renewcommand{\sec}{\mathbb{S}ec}

\DeclareMathOperator{\Sym}{Sym}
\DeclareMathOperator{\Cox}{Cox}

\DeclareMathOperator{\mov}{mov}

\renewcommand{\P}{\mathbb{P}}

\newtheorem{thm}{Theorem}[section]

\newtheorem{Question}[thm]{Question}
\newtheorem{Lemma}[thm]{Lemma}
\newtheorem{Proposition}[thm]{Proposition}

\newtheorem{Corollary}[thm]{Corollary}

\theoremstyle{definition}

\newtheorem{Definition}[thm]{Definition}
\newtheorem{Remark}[thm]{Remark}

\newtheorem{Notation}[thm]{Notation}

\newtheorem{Construction}[thm]{Construction}

\newtheorem*{thm1}{Theorem A}
\newtheorem*{thm2}{Theorem B}

\hypersetup{pdfpagemode=UseNone}
\hypersetup{pdfstartview=FitH}

\begin{document}
\title{The Lefschetz principle in birational geometry: birational twin varieties}

\author[C\'esar Lozano Huerta]{C\'esar Lozano Huerta}
\address{\sc C\'esar Lozano Huerta\\
Universidad Nacional Aut\'onoma de M\'exico, Instituto de Matem\'aticas\\
68000, Leon 2, Centro, Oaxaca, M\'exico}
\email{lozano@im.unam.mx}

\author[Alex Massarenti]{Alex Massarenti}
\address{\sc Alex Massarenti\\ Dipartimento di Matematica e Informatica, Universit\`a di Ferrara, Via Machiavelli 30, 44121 Ferrara, Italy\newline
\indent Instituto de Matem\'atica e Estat\'istica, Universidade Federal Fluminense, Campus Gragoat\'a, Rua Alexandre Moura 8 - S\~ao Domingos\\
24210-200 Niter\'oi, Rio de Janeiro\\ Brazil}
\email{alex.massarenti@unife.it, alexmassarenti@id.uff.br}

\date{\today}

\begin{abstract} Inspired by the Weak Lefschetz Principle, we study when a smooth projective variety fully determines the birational geometry of some of its subvarieties. In particular, we consider the natural embedding of the space of complete quadrics into the space of complete collineations and we observe that their birational geometry, from the point of view of Mori theory, fully determines each other. When two varieties are related in this way, we call them birational twins. We explore this notion and its various flavors for other embeddings between Mori dream spaces.
\end{abstract}

\maketitle 

\setcounter{tocdepth}{1}

\tableofcontents

\section{Introduction}
\noindent  The celebrated Lefschetz hyperplane theorem asserts that if there is a smooth complex projective variety $Y$ with the inclusion of a smooth ample hypersurface $i:X\rightarrow Y$, then the ambient variety $Y$ often fully determines many of the topological invariants of $X$. Such invariants include homology, cohomology (compatible with Hodge structures) and homotopy groups among others. This phenomenon leads to the question: what are some algebro-geometric invariants to which a result of this type applies? In other words, we are asking if the Weak Lefschetz Principle holds beyond topological invariants.

The purpose of this paper is to investigate the Weak Lefschetz Principle in the context of birational geometry. In other words, if we consider two varieties $X,Y$ as well as an embedding between them $i: X\rightarrow Y$, then we ask about the birational invariants $X$ which are determined by those of $Y$. Such invariants may include effective and ample cones, or finer information as we now explain.

Let us recall that a normal $\mathbb{Q}$-factorial projective variety $X$ with finitely generated Picard group is a Mori dream space if and only if its Cox ring is finitely generated \cite[Proposition 2.9]{HK00}. The birational geometry of these varieties is encoded in this ring and we may ask whether the Weak Lefschetz Principle applies to it. This question has recently been studied in the case of hypersurfaces \cite{HLW01,JOW,Ottem,Szen}, and for Mori dream spaces of Picard rank two \cite{LMR18}. In addition to considering the Cox ring, let us recall that the pseudo-effective cone $\overline{\Eff}(X)$ of a projective variety $X$ with irregularity zero, such as a Mori dream space, can be decomposed into chambers depending on the stable base locus of the corresponding linear series. This decomposition is called \textit{stable base locus decomposition} and in general it is coarser than the so-called \textit{Mori chamber decomposition}; which encodes the isomorphism type of the birational models of $X$. Mori dream spaces and these decompositions have been widely studied by many authors recently, for instance see \cite{CT06, Mu01, AM16, MR18} and the references therein.   

We may also ask whether the previous two decompositions satisfy the Weak Lefschetz Principle. We then summarize the previous discussion in the following definition. Let us denote by $\SBLD(X)$ the stable base locus decomposition of $\overline{\Eff}(X)$.

\begin{Definition}\label{Pre}
Let $X,Y$ be $\mathbb{Q}$-factorial projective varieties and $i:X\rightarrow Y$ be an embedding. These varieties are said \textit{Lefschetz divisorially equivalent} if the pull-back $i^*:\Pic(Y)\rightarrow\Pic(X)$ induces an isomorphism such that
$$i^*\Eff(Y) = \Eff(X),\quad  i^*\Mov(Y) = \Mov(X),\quad  i^*\Nef(Y) = \Nef(X)$$
Furthermore, $X$ and $Y$ are said \textit{strongly Lefschetz divisorially equivalent} if they are Lefschetz  divisorially equivalent and in addition $i^{*}\SBLD(Y) = \SBLD(X)$, with $h^1(X,\mathcal{O}_X) = h^1(Y,\mathcal{O}_Y)=0$. 
\end{Definition}

It is an immediate consequence of the Lefschetz hyperplane theorem that a smooth variety $Y$, of Picard rank one and dimension at least four, is Lefschetz divisorially equivalent to any of its smooth ample divisors. If the assumption that the Picard rank is one is removed in this case, B. Hassett, H-W. Lin and C-L. Wang \cite{HLW01} exhibited an example of a variety and an ample divisor $D\subset Y$ whose Picard groups are isomorphic by the Lefschetz hyperplane theorem, but such that the nef cone is not preserved; hence $D$ and $Y$ are not Lefschetz divisorially equivalent.

The previous definition is not restricted to subvarieties of codimension one. In fact, the results in this paper involve classical spaces and subvarieties of high codimension, such that the birational geometry of both objects can now be linked using Definition \ref{Pre}. Let us recall those spaces so we can state our results precisely.

Let $V,W$ be two $K$-vector spaces of dimensions $n+1$ and $m+1$, respectively, with $n\leq m$. The field $K$ is algebraically closed of characteristic zero. We will denote by $\mathcal{X}(n,m)$ the space of complete collineations $V\rightarrow W$ and by $\mathcal{Q}(n)$ the space of complete $(n-1)$-dimensional quadrics of $V$. In \cite{Va82}, \cite{Va84}, I. Vainsencher showed that these spaces can be understood as sequences of blow-ups along the subvariety parametrizing rank one matrices and the strict transforms of its secant varieties. Such blowups are central characters of this paper, so let us be more precise about them. Given an irreducible and reduced non-degenerate variety $X\subset\P^N$, and a positive integer $h\leq N$, the \textit{$h$-secant variety} $\sec_h(X)$ of $X$ is the subvariety of $\P^N$ obtained as the closure of the union of all $(h-1)$-planes spanned by $h$ general points of $X$. Spaces of matrices and symmetric matrices admit a natural stratification dictated by the rank and observe that a general point of the $h$-secant variety of a Segre, or a Veronese, corresponds to a matrix of rank $h$. More precisely, let $\mathbb{P}^N$ be the projective space parametrizing $(n+1)\times (n+1)$ matrices modulo scalars, $\mathbb{P}^{N_{+}}$ the subspace of symmetric matrices, $\mathcal{S}\subset\mathbb{P}^N$ the Segre variety, and $\mathcal{V}\subset\mathbb{P}^{N_{+}}$ the Veronese variety. Set $\mathcal{X}(n):=\mathcal{X}(n,n)$ and denote by $\mathcal{X}(n)_i$, $\mathcal{Q}(n)_i$ respectively the varieties obtained at the step $i$ in Vainsencher's constructions \ref{ccc}; which roughly says that $\mathcal{X}(n)_i$, $\mathcal{Q}(n)_i$ are the blow-ups respectively of $\P^N$ and $\P^{N_{+}}$ along all the $j$-secant varieties of $\mathcal{S}$ and $\mathcal{V}$ for $j\le i$. Since $\sec_h(\mathcal{V}) = \sec_h(\mathcal{S})\cap\mathbb{P}^{N_{+}}$, the natural inclusion $\mathbb{P}^{N_{+}}\hookrightarrow\mathbb{P}^{N}$ lifts to an embedding $\mathcal{Q}(n)_i\hookrightarrow\mathcal{X}(n)_i$ for any $i$.

These families of spaces $\mathcal{X}(n)_i$ and $\mathcal{Q}(n)_i$ are very different in may aspects, for instance their dimensions are different. However, when it comes to certain birational invariants they are indistinguishable. Indeed, by Proposition \ref{theff} and Lemma \ref{lemma_mov}, we have our first result.

\begin{thm1}\label{thm1}
\textit{
The varieties $\mathcal{X}(n)_i$ and $\mathcal{Q}(n)_i$ are Lefschetz divisorially equivalent via the usual embedding for any $i$. In particular, the spaces of complete collineations $\mathcal{X}(n)$ and quadrics $\mathcal{Q}(n)$ are Lefschetz divisorially equivalent via the usual embedding.}
\end{thm1}

A more thorough version of the Weak Lefschetz Principle in our context would be the following: a small modification of the ambient variety $Y$ restricts to a small modification, possibly an isomorphism, of the subvariety $X$, and all the small modifications of $X$ can be obtained in this way. This is stronger than determining invariants and it is the content of the following definition. We will test this definition on some Mori dream spaces. The Mori chamber decomposition of $\Eff(X)$, of a Mori dream space $X$, will be denoted by $\MCD(X)$. 

\begin{Definition}\label{main_def1_intro}
Let $X,Y$ be as in the previous definition, we say that they are \textit{birational twins} if they are Lefschetz divisorially equivalent Mori dream spaces and in addition $i^{*}\MCD(Y) = \MCD(X)$.
Also, they are said to be \textit{strong birational twins} if they are birational twins and in addition $i^{*}\SBLD(Y) = \SBLD(X)$.
\end{Definition}

Our next result shows that the relationship between the birational geometry of complete collineations and complete quadrics, in some cases, goes deeper than Theorem A. The following result follows from Theorem \ref{MCD_main} and Corollary \ref{CQ3}. 

\begin{thm2}\label{thm2}
\textit{
The varieties $\mathcal{X}(n)_3$ and $\mathcal{Q}(n)_3$ are birational twins for any $n\geq 1$. Furthermore, $\mathcal{X}(3)$ and $\mathcal{Q}(3)$ are strong birational twins.}
\end{thm2}

In general, the spaces $\mathcal{X}(n)_i,\mathcal{Q}(n)_i$ are Mori dream spaces for any $i$; we can actually count the number of chambers in the Mori chamber decomposition in various cases. Furthermore, the ample, movable and effective cones are preserved by Theorem A. However, we do not know if the structure of the Mori chambers is also preserved. In other words, if Theorem B holds in general.

\begin{Question}
Are $\mathcal{X}(n)_i$ and $\mathcal{Q}(n)_i$ birational twins for $n\geq 4$ and any $i\ge 4$?
\end{Question}

We organize this paper as follows. In Section \ref{bir_twin}, we discuss Definitions \ref{Pre} and \ref{main_def1_intro} and the relations among the various flavors of these definitions. In particular, Theorem \ref{thmex} provides examples of Mori dream spaces that are strongly Lefschetz divisorially equivalent but not birational twins, and of Mori dream spaces that are birational twins but not strong birational twins. In Section \ref{sec1}, we introduce the spaces of complete collineations and quadrics, and we study their birational geometry from a Mori theoretic viewpoint. Finally, in Section \ref{sec_flip} we recall the modular description of the unique flip of $\mathcal{Q}(3)$ given in \cite[Section 5.2]{Ce15}, and based on it, we give a conjectural description of the unique flip of $\mathcal{X}(3)$.    

\subsection*{Acknowledgments}
The first named author is currently a CONACYT-Research Fellow and during the preparation of this article he was partly supported by the CONACYT Grant CB-2015/253061. The second named author is a member of the Gruppo Nazionale per le Strutture Algebriche, Geometriche e le loro Applicazioni of the Istituto Nazionale di Alta Matematica "F. Severi" (GNSAGA-INDAM). We thank the organizers of the conferences \textit{Birational Geometry and Moduli Spaces}, held in Rome, Italy, and \textit{Moduli Spaces in Algebraic Geometry and Applications}, held in Campinas, Brazil, where part of this work was carried out. We have benefited from discussion with G. Reyes-Ahumada. We thank the referee for the comments that helped us to improve the paper.

\section{Lefschetz divisorially equivalent varieties and birational twins}\label{bir_twin}
\noindent
This section contains preliminaries and explores various aspects of definitions \ref{Pre} and \ref{main_def1_intro}. In particular, Theorem \ref{thmex} and Proposition \ref{strongtwins} are the main results of this section.

Let $X$ be a normal projective variety over an algebraically closed field of characteristic zero. We denote by $N^1(X)$ the real vector space of $\mathbb{R}$-Cartier divisors modulo numerical equivalence. 
The \emph{nef cone} of $X$ is the closed convex cone $\Nef(X)\subset N^1(X)$ generated by classes of nef divisors. The \emph{movable cone} of $X$ is the convex cone $\Mov(X)\subset N^1(X)$ generated by classes of \emph{movable divisors}; these are Cartier divisors whose stable base locus has codimension at least two in $X$.
The \emph{effective cone} of $X$ is the convex cone $\Eff(X)\subset N^1(X)$ generated by classes of \emph{effective divisors}. We have inclusions $\Nef(X)\ \subset \ \overline{\Mov(X)}\ \subset \ \overline{\Eff(X)}$. 

We will denote by $N_1(X)$ the real vector space of numerical equivalence classes of $1$-cycles on $X$. The closure of the cone in $N_1(X)$ generated by the classes of irreducible curves in $X$ is called the \textit{Mori cone} of $X$, we will denote it by $\NE(X)$. 

A class $[C]\in N_1(X)$ is called \textit{moving} if the curves in $X$ of class $[C]$ cover a dense open subset of $X$. The closure of the cone in $N_1(X)$ generated by classes of moving curves in $X$ is called the \textit{moving cone} of $X$ and we will denote it by $\mov(X)$.  
We refer to \cite[Chapter 1]{De01} for a comprehensive treatment of these topics. 

We say that a birational map  $f: X \dasharrow X'$ to a normal projective variety $X'$  is a \emph{birational contraction} if its inverse does not contract any divisor. 
We say that it is a \emph{small $\mathbb{Q}$-factorial modification} 
if $X'$ is $\mathbb{Q}$-factorial  and $f$ is an isomorphism in codimension one.
If  $f: X \dasharrow X'$ is a small $\mathbb{Q}$-factorial modification, then 
the natural pull-back map $f^*:N^1(X')\to N^1(X)$ sends $\Mov(X')$ and $\Eff(X')$
isomorphically onto $\Mov(X)$ and $\Eff(X)$, respectively.
In particular, we have $f^*(\Nef(X'))\subset \overline{\Mov(X)}$.

\begin{Definition}\label{def:MDS} 
A normal projective $\mathbb{Q}$-factorial variety $X$ is called a \emph{Mori dream space}
if the following conditions hold:
\begin{enumerate}
\item[-] $\Pic{(X)}$ is finitely generated, or equivalently $h^1(X,\mathcal{O}_X)=0$,
\item[-] $\Nef{(X)}$ is generated by the classes of finitely many semi-ample divisors,
\item[-] there is a finite collection of small $\mathbb{Q}$-factorial modifications
 $f_i: X \dasharrow X_i$, such that each $X_i$ satisfies the second condition above, and $
 \Mov{(X)} \ = \ \bigcup_i \  f_i^*(\Nef{(X_i)})$.
\end{enumerate}
\end{Definition}

By \cite[Corollary 1.3.2]{BCHM10} smooth Fano varieties are Mori dream spaces. In fact, there is a larger class of varieties called log Fano varieties which are Mori dream spaces as well. By work of M. Brion \cite{Br93} we have that $\mathbb{Q}$-factorial spherical varieties are Mori dream spaces. An alternative proof of this result can be found in \cite[Section 4]{Pe14}. 

The collection of all faces of all cones $f_i^*(\Nef{(X_i)})$ in Definition \ref{def:MDS} forms a fan which is supported on $\Mov(X)$.
If two maximal cones of this fan, say $f_i^*(\Nef{(X_i)})$ and $f_j^*(\Nef{(X_j)})$, meet along a facet,
then there exist a normal projective variety $Y$, a small modification $\varphi:X_i\dasharrow X_j$, and $h_i:X_i\rightarrow Y$ and $h_j:X_j\rightarrow Y$ small birational morphism of relative Picard number one such that $h_j\circ\varphi = h_i$. The fan structure on $\Mov(X)$ can be extended to a fan supported on $\Eff(X)$ as follows. 

\begin{Definition}\label{MCD}
Let $X$ be a Mori dream space. We describe a fan structure on the effective cone $\Eff(X)$, called the \emph{Mori chamber decomposition}.
We refer to \cite[Proposition 1.11]{HK00} and \cite[Section 2.2]{Ok16} for details.
There are finitely many birational contractions from $X$ to Mori dream spaces, denoted by $g_i:X\dasharrow Y_i$.
The set $\Exc(g_i)$ of exceptional prime divisors of $g_i$ has cardinality $\rho(X/Y_i)=\rho(X)-\rho(Y_i)$.
The maximal cones $\mathcal{C}$ of the Mori chamber decomposition of $\Eff(X)$ are of the form: $\mathcal{C}_i \ = \left\langle g_i^*\big(\Nef(Y_i)\big) , \Exc(g_i) \right\rangle$. We call $\mathcal{C}_i$ or its interior $\mathcal{C}_i^{^\circ}$ a \emph{maximal chamber} of $\Eff(X)$.
\end{Definition}

\begin{Definition}
Let $X$ be a normal projective variety with finitely generated divisor class group $\Cl(X) := {\rm WDiv}(X)/{\rm PDiv}(X)$, in particular $h^1(X,\mathcal O_X)=0$. The \textit{Cox sheaf} and {\em Cox ring}
of $X$ are defined as
\[
 \mathcal R := \bigoplus_{[D]\in \Cl(X)}\mathcal{O}_X(D)
 \qquad
 \qquad
 \Cox(X) := \Gamma(X,\mathcal R)
\]
\end{Definition}

Recall that $\mathcal R$ is a sheaf of
${\rm Cl}(X)$-graded $\mathcal O_X$-algebras, 
whose multiplication maps are discussed in
~\cite[Section 1.4]{ADHL15}. In case the divisor
class group is torsion-free, one can just 
take the direct sum over a subgroup of
${\rm WDiv}(X)$, isomorphic to ${\rm Cl}(X)$ 
via the quotient map, getting immediately a sheaf of $\mathcal O_X$-algebras.
Denote by $\widehat X$ the 
relative spectrum of $\mathcal R$ 
and by $\overline X$ the spectrum
of $\Cox(X)$. The ${\rm Cl}(X)$-grading
induces an action of the quasitorus 
$H_X := {\rm Spec}\,\mathbb C[{\rm Cl}(X)]$
on both spaces. The inclusion 
$\mathcal O_X\to\mathcal R$ induces 
a good quotient $p_X\colon \widehat X\to X$
with respect to this action.
Summarizing, we have the following	
diagram
  \[
  \begin{tikzpicture}[xscale=0.65,yscale=-1.2]
    \node (A0_0) at (0, 0) {$\widehat{X}$};
    \node (A0_1) at (1, 0) {$\subseteq\overline{X}$};
    \node (A1_0) at (0, 1) {$X$};
    \path (A0_0) edge [->,swap]node [auto] {$\scriptstyle{p_X}$} (A1_0);
  \end{tikzpicture}
  \]
and we will refer to it as the {\em Cox 
construction} of $X$. In case $\Cox(X)$
is a finitely generated algebra, the complement
of $\widehat X$ in the affine variety 
$\overline X$ has codimension $\geq 2$.
This subvariety is the {\em irrelevant locus}
and its defining ideal is the {\em irrelevant
ideal} $\mathcal J_{\rm irr}(X) \subseteq 
\Cox(X)$.

Let $X$ be a normal $\mathbb{Q}$-factorial projective variety, and let $D$ be an effective $\mathbb{Q}$-divisor on $X$. The stable base locus $\textbf{B}(D)$ of $D$ is the set-theoretic intersection of the base loci of the complete linear systems $|sD|$ for all positive integers $s$ such that $sD$ is integral. In other words,
$$\textbf{B}(D) = \bigcap_{s > 0}B(sD)$$
Since stable base loci do not behave well with respect to numerical equivalence, we will assume that $h^{1}(X,\mathcal{O}_X)=0$. This implies that linear and numerical equivalence of $\mathbb{Q}$-divisors coincide. 

Since numerically equivalent $\mathbb{Q}$-divisors on $X$ have the same stable base locus, then $\overline{\Eff}(X)$ the pseudo-effective cone  of $X$ can be decomposed into chambers depending on the stable base locus of the corresponding linear series. This decomposition is called \textit{stable base locus decomposition}. 

If $X$ is a Mori dream space, satisfying then the condition $h^1(X,\mathcal{O}_X)=0$, determining the stable base locus decomposition of $\Eff(X)$ is a first step in order to compute its Mori chamber decomposition. 

\begin{Remark}\label{SBLMC}
Recall that two divisors $D_1,D_2$ are said to be \textit{Mori equivalent} if $\textbf{B}(D_1) = \textbf{B}(D_2)$ and the following diagram of rational maps is commutative
   \[
  \begin{tikzpicture}[xscale=1.5,yscale=-1.2]
    \node (A0_1) at (1, 0) {$X$};
    \node (A1_0) at (0, 1) {$X(D_1)$};
    \node (A1_2) at (2, 1) {$X(D_2)$};
    \path (A1_0) edge [->]node [auto] {$\scriptstyle{}$} node [rotate=180,sloped] {$\scriptstyle{\widetilde{\ \ \ }}$} (A1_2);
    \path (A0_1) edge [->,dashed]node [auto] {$\scriptstyle{\phi_{D_2}}$} (A1_2);
    \path (A0_1) edge [->,swap, dashed]node [auto] {$\scriptstyle{\phi_{D_1}}$} (A1_0);
  \end{tikzpicture}
  \]
where the horizontal arrow is an isomorphism. Therefore, the Mori chamber decomposition is a refinement of the stable base locus decomposition. 
\end{Remark}

Let $X$ be a Mori dream space with Cox ring
$\Cox(X)$ and grading matrix $Q$. The matrix $Q$ 
defines a surjection
\[
 Q\colon E\to{\rm Cl}(X)
\]
from a free, finitely generated, 
abelian group $E$ to the divisor 
class group of $X$.
Denote by $\gamma$ the positive quadrant of 
$E_{\mathbb Q} := E\otimes_{\mathbb Z}\mathbb Q$.
Let $e_1,\dots,e_r$ be the canonical 
basis of $E_{\mathbb Q}$.
Given a face $\gamma_0\preceq\gamma$, we say that $i\in\{1,\dots,r\}$ is a cone index 
of $\gamma_0$ if $e_i\in \gamma_0$.
The face $\gamma_0$ is an {\em $\mathfrak F$-face}
if there exists a point of $\overline X 
= {\rm Spec}(\Cox(X))$ whose $i$-th coordinate
is non-zero exactly when $i$ is a cone index 
of $\gamma_0$ \cite[Construction 3.3.1.1]{ADHL15}.
The set of these points is denoted by 
$\overline X(\gamma_0)$.
Given the Cox construction of $X$ we denote by $X(\gamma_0)\subseteq X$ the 
image of $\overline X(\gamma_0)$, and given an $\mathfrak F$-face $\gamma_0$
its image $Q(\gamma_0)\subseteq {\rm Cl}(X)_{\mathbb Q}$
is an {\em orbit cone} of $X$. The set of all orbit
cones of $X$ is denoted by $\Omega$.
Accordingly to \cite[Definition 3.1.2.6]{ADHL15}
an effective class $w\in {\rm Cl}(X)$ defines the {\em GIT
chamber}
\stepcounter{thm}
\begin{equation}\label{gitch}
\lambda(w):=\bigcap_{\{\omega\in\Omega\, :\, w\in\omega\}}\omega
\end{equation}
If $w$ is an ample class of $X$, then the corresponding
GIT chamber is the semi-ample cone of $X$.
The variety $X$ can be reconstructed from
the pair $(\Cox(X),\Phi)$ formed by
the Cox ring together with a {\em bunch of cones},
consisting of certain subsets of the orbit cones
~\cite[Definition 3.1.3.2]{ADHL15}. 



We now explore different aspects of definitions \ref{Pre} and \ref{main_def1_intro}. We exhibit two varieties which are birational twins but fail to be strong birational twins. We do this inspired by \cite[Theorem 1.1.]{LMR18}. We also exhibit two Mori dream spaces which are strongly Lefschetz divisorially equivalent but fail to be birational twins. 

\begin{thm}\label{thmex}
Let $Z$ be the toric variety with Cox ring
$K[T_1,\dots,T_{11}]$ whose 
grading matrix and irrelevant ideal are
the following
\[
 Q = \begin{bmatrix}
  1&1&2&2&2&2&1&0&0&0&0\\
  0&0&1&1&1&1&2&1&1&1&1
 \end{bmatrix}
 \qquad
  \mathcal J_{\rm irr}(Z) 
  = 
  \langle T_1,T_2\rangle
  \cap
  \langle T_3,\dots,T_{11}\rangle
\]
and let $F = T_3T_8+T_4T_9+T_1T_7, G = T_5T_{9}+T_6T_{11}+T_2T_7$. Then the ring $\frac{K[T_1,\dots,T_{11}]}{(F,G)}$ is the Cox ring of a projective normal $\mathbb{Q}$-factorial Mori dream space $X\subset Z$ of Picard rank two. The varieties $X$ and $Z$ are birational twins but fail to be strong birational twins.

Furthermore, the ring $\frac{K[T_1,\dots,T_{11}]}{(F,G,T_7+T_1T_{11}^2+T_2T_8T_{11}+T_2T_{10}^2+T_2T_8T_9)}$ is the Cox ring of a projective normal $\mathbb{Q}$-factorial Mori dream space $Y\subset X$ of Picard rank two. The varieties $Y$ and $X$ are strongly Lefschetz divisorially equivalent but fail to be birational twins.  
\end{thm}
\begin{proof}
Let us denote by $w_i\in\mathbb Z^2$ the 
degree of $T_i$. We will denote by $V(T_{i_1},\dots,T_{i_r}):=\{T_{i_1}=\dots =T_{i_r}=0\}$. The following picture represents the degrees of the generators of the Cox ring
\begin{center}
 \begin{tikzpicture}
   \path [fill=lightgray] (0,0) -- (1,0) -- (2,1);
  \foreach \x/\y in {1/0/1,0/1,2/1,1/2}
   {\draw[fill] (\x,\y) circle [radius=0.05];
    \draw[-] (0,0) -- (\x,\y);}
   \node[below right] at (1,0) {$w_1,w_2$};
   \node[above left] at (0,1) {$w_8,\dots,w_{11}$};
   \node[above right] at (2,1) {$w_3,\dots,w_6$};
   \node[above right] at (1,2) {$w_7$};
   \node at (0.3,1) {$\lambda'$};
   \node at (0.7,0.7) {$\lambda$};
   \node at (1,0.25) {$\lambda_A$};
 \end{tikzpicture}
\end{center}
where $\Eff(Z)=\Mov(Z) = \lambda_A\cup\lambda\cup\lambda'$ and $\Nef(Z) = \lambda_A$. Let $\overline Z = K^{11}$ be the spectrum of the Cox ring of $Z$ and let $\overline X$ be the affine subvariety 
defined by $\overline{X} = \{F = G = 0\}\subset\overline{Z}$.

A standard computation shows that $\overline{X}$ is irreducible and $\codim_{\overline{X}}(\Sing(\overline{X})) \geq 2$, and then Serre's criterion on normality yields that $\overline{X}$ is normal. Let $p_Z\colon\widehat Z\rightarrow Z$ be the characteristic space morphism
of $Z$, and let $\widehat X := \overline X \cap
\widehat Z$. Note that the irrelevant locus with respect to the chamber $\lambda_A$ has two components given by $\Gamma_1 = \{T_1 = T_2 = 0\}$ and $\Gamma_2 =\{T_3 = \dots = T_{11} = 0\}$. Hence $\widehat{Z} = K^{11}\setminus\{\Gamma_1\cup\Gamma_2\}$.

The image of $\widehat X$ via $p_Z$ is a subvariety $X$ of $Z$, let $i:X\rightarrow Z$ be the inclusion. Since $\overline{X}$ is irreducible and normal, and $X$ is a GIT quotient of $\overline{X}$ by a reductive group \cite[Theorem 1.24 (vi)]{Br10} yields that $X$ is irreducible and normal as well. 

Let $Z'$ be the smooth locus of $Z$. Note that $Z'$ contains the open subset $Z''$ of $Z$ obtained by removing the union of all the toric subvarieties of the form $p_Z(V(T_i,T_j))$. Since the Zariski
closure of $p_Z^{-1}(X\cap Z'')$ in $\widehat Z$ is $\widehat X$ we conclude that $\widehat X$ is the Zariski closure of $p_Z^{-1}(X\cap Z')$ in $\widehat Z$.

Note that, according to the grading matrix $Q$, the $K^{*}\times K^{*}$ action on $\overline{X}$ is given by 
$$
\begin{array}{ccc}
(K^{*}\times K^{*})\times\overline{X} & \longrightarrow & \overline{X}\\
((\lambda,\mu),(T_{1},...,T_{11})) & \longmapsto & (\lambda T_{1},\lambda T_2,\lambda^2\mu T_3,\dots, \lambda\mu^2 T_7,\mu T_8,\dots,\mu T_{11})
\end{array}
$$ 
Therefore, if $T_2\neq 0, T_8\neq 0$ we may set $\lambda = \mu = 1$, and write $T_3 = -\frac{T_4T_9+T_1T_7}{T_8}, T_7 = -\frac{T_5T_{9}+T_6T_{11}}{T_2}$. Therefore, if we remove the images of $V(T_2)\cup V(T_8)$ from $X$, the resulting variety is isomorphic to an affine space. Therefore, ${\rm Cl}(X)$ is generated by the classes of the images of the two irreducible divisors $V(T_i)\cap \overline X$, with $i\in\{2,8\}$, and $\rho(X)\leq 2$. Moreover, crossing the wall corresponding to $w_1,w_2$ we get a morphism $f:Z\rightarrow\mathbb{P}^1$. Furthermore, $X\subset Z$ is not contained in any fiber of $f$, and hence $f$ restricts to a surjective morphism $f_{|X}:X\rightarrow\mathbb{P}^1$. This forces $\rho(X)\geq 2$. Finally, we conclude that the images of $V(T_2),V(T_8)$ form a basis of $\Cl(X)$. So $\rho(X)=2$ and the pull-back $i^{*}:\Cl(Z)\rightarrow\Cl(X)$, induced by the inclusion, is an isomorphism. Now, note that $\Gamma_2\subset\overline{X}$ has codimension greater that two in $\overline{X}$, and $\Gamma_1\cap\overline{X} = \{T_1 = T_2 = T_3T_8+T_4T_9 = T_5T_{9}+T_6T_{11}=0\}$ has codimension two in $\overline{X}$. Hence $\codim_{\overline{X}}(\overline X\setminus\widehat X)\geq 2$ and \cite[Corollary 4.1.1.5]{ADHL15} yields that the Cox ring of $X$ is
$$
\Cox(X) = \frac{K[T_1,\dots,T_{11}]}{I(\overline X)}
$$ 
In general, if $X$ is a Mori dream space such that $\Cl(X)$ has rank two we can fix a total order on the classes in the effective cone $w\leq w'$ if $w$ is on the left of $w'$. Given two convex cones $\lambda,\lambda'$ contained in the effective cone we will write $\lambda\leq \lambda'$ if $w\leq w'$ for any $w\in\lambda$ and $w'\in\lambda$. Note that
$$
\begin{array}{l}
V(f_i\, :\, w_i\geq\lambda_A) = \{T_1=T_2=\widetilde{F}=\widetilde{G}=0\}\\[2pt]\end{array}$$
where $\widetilde{F} = T_3T_8+T_4T_9,\widetilde{G} = T_5T_{9}+T_6T_{11}$. Furthermore
$$\begin{array}{l}
V(f_i\, :\, w_i\leq\lambda) = \{T_7=T_8=T_9=T_{10}=T_{11}=0\};\\[2pt]
V(f_i\, :\, w_i\leq\lambda') = \{T_1=T_2=T_8=T_9=T_{10}=T_{11}=0\}\cup\{T_7=T_8=T_9=T_{10}=T_{11}=0\}.
\end{array}
$$
So $V(f_i\, :\, w_i\geq\lambda_A)\supset \{T_1=T_2=T_8=T_9=T_{10}=T_{11}=0\}$ contains a component of the set $V(f_i\, :\, w_i\leq\lambda')$. Finally, ~\cite[Lemmas 4.3, 4.4, 4.5]{LMR18} yields that $\MCD(X)=i^{*}\MCD(Z)$, and that $\SBLD(X)$ has just two chambers while $\SBLD(Z)=\MCD(Z)$ has three chambers. 

Let $\overline{Y} \subset \overline{Z} = K^{11}$ be the subvariety cut out by 
$$\overline{Y}=\{F=G=T_7+T_1T_{11}^2+T_2T_8T_{11}+T_2T_{10}^2+T_2T_8T_9=0\}$$ 
We have that $\overline{Y}$ is irreducible of dimension eight and $\codim_{\overline{Y}}(\Sing(\overline{Y})) \geq 2$, so Serre's criterion on normality yields that $\overline{Y}$ is normal, and \cite[Theorem 1.24 (vi)]{Br10} yields that the GIT quotient $Y$ is irreducible and normal as well.

If $T_2\neq 0, T_8\neq 0$ we may write $T_3 = -\frac{T_4T_9+T_1T_7}{T_8}, T_7 = -\frac{T_5T_{9}+T_6T_{11}}{T_2}$ and $T_{9} = -\frac{T_2T_{10}^2+T_2T_8T_{11}+T_1T_{11}^2+T_7}{T_2T_8}$. Therefore, $\overline{Y}\setminus (V(T_2)\cup V(T_8))$ is isomorphic to an affine space and ${\rm Cl}(Y)$ is generated by the classes of the images of the two irreducible divisors $V(T_i)\cap \overline Y$, with $i\in\{2,8\}$, and $\rho(Y)\leq 2$. Again crossing the wall corresponding to $w_1,w_2$ we get a morphism $f:Z\rightarrow\mathbb{P}^1$ and since $Y\subset Z$ is not contained in any fiber of $f$ this morphism restricts to a surjective morphism $f_{|Y}:Y\rightarrow\mathbb{P}^1$. This forces $\rho(Y)\geq 2$. We conclude that the images of $V(T_2),V(T_8)$ form a basis of $\Cl(Y)$. So $\rho(Y)=2$ and the pull-back $j^{*}:\Cl(X)\rightarrow\Cl(Y)$, induced by the inclusion $j:Y\rightarrow X$, is an isomorphism. As before $\Gamma_2\subset\overline{Y}$ has codimension greater that two in $\overline{Y}$, and $\Gamma_1\cap\overline{Y}$ has codimension two in $\overline{Y}$. Hence \cite[Corollary 4.1.1.5]{ADHL15} yields that the Cox ring of $Y$ is
$$
\Cox(Y) = \frac{K[T_1,\dots,T_{11}]}{I(\overline{Y})}
$$ 
Now, for the the Mori dream space $Y$ we have   
$$
\begin{array}{l}
V(f_i\, :\, w_i\leq\lambda) = V(f_i\, :\, w_i\leq\lambda') = \{T_7=T_8=T_9=T_{10}=T_{11}=0\}; \\ 
V(f_i\, :\, w_i\geq\lambda) = V(f_i\, :\, w_i\geq\lambda') = \{T_1=T_2=T_3=T_{4}=T_{5}=T_6 = T_7 = 0\}.
\end{array} 
$$
Hence ~\cite[Lemma 4.3]{LMR18} yields that $\MCD(Y)$ has two chambers while $\MCD(X)=i^{*}\MCD(Z)$ has three chambers.
\end{proof}

\begin{Remark} 
Note that any pair of varieties among the Mori dream spaces $Y\subset X\subset Z$ in Theorem \ref{thmex} gives a pair of Lefschetz divisorially equivalent varieties. On the other hand, $i^{*}\SBLD(Z)\neq \SBLD(X)$ and $(i\circ j)^{*}\MCD(Z)\neq \MCD(Y)$. Therefore, Lefschetz divisorially equivalent does not imply strongly Lefschetz divisorially equivalent nor birational twins.  
\end{Remark}

However, when a Mori dream space inherits the stable base locus decomposition from an ambient toric variety something more can be said. 

\begin{Remark}\label{toric}
Given a Mori dream space $X$, then there is an embedding $i:X\rightarrow \mathcal{T}_X$ into a simplicial projective toric variety $\mathcal{T}_X$ such that $i^{*}:\Pic(\mathcal{T}_X)\rightarrow \Pic(X)$ is an isomorphism which induces an isomorphism $\Eff(\mathcal{T}_X)\rightarrow \Eff(X)$, \cite[Proposition 2.11]{HK00}. Furthermore, the Mori chamber decomposition of $\Eff(\mathcal{T}_X)$ is a refinement of the Mori chamber decomposition of $\Eff(X)$. Indeed, if $\Cox(X) \cong \frac{K[T_1,\dots,T_s]}{I}$ where the $T_i$ are homogeneous generators with non-trivial effective $\Pic(X)$-degrees then $\Cox(\mathcal{T}_X)\cong K[T_1,\dots,T_s]$.
\end{Remark}

\begin{Proposition}\label{strongtwins}
Let $X\rightarrow \mathcal{T}_X$ be a Mori dream space embedded in its canonical toric embedding. Assume that both $X$ and $\mathcal{T}_X$ have Picard rank two. If $X$ and $\mathcal{T}_X$ are strongly Lefschetz divisorially equivalent then they are strong birational twins.
\end{Proposition}
\begin{proof}
By Remark \ref{toric} $\MCD(\mathcal{T}_X)$ refines $\MCD(X)$, and moreover $\MCD(X)$ refines $\SBLD(X)$. By our hypothesis $\SBLD(X) = i^{*}\SBLD(\mathcal{T}_X)$, where $i:X\rightarrow \mathcal{T}_X$ is the inclusion. Now, since $\mathcal{T}_X$ is a toric projective variety of Picard rank two ~\cite[Theorem 1.2]{LMR18} yields $\SBLD(\mathcal{T}_X) = \MCD(\mathcal{T}_X)$. Summing up, $\MCD(\mathcal{T}_X) = \SBLD(\mathcal{T}_X)$ refines $\MCD(X)$ which in turn refines $\SBLD(X) = \SBLD(\mathcal{T}_X) = \MCD(\mathcal{T}_X)$, and this forces $\MCD(X) = i^{*}\MCD(\mathcal{T}_X)$.
\end{proof}

\section{Birational twin varieties: complete collineations and quadrics}\label{sec1}
\noindent
This section is the core of the paper. It contains the definitions and crucial properties of the spaces of complete collineations and complete quadrics. It also contains the details of the proofs of Theorem A and Theorem B listed in the introduction.

Let $V$ and $W$ be $K$-vector spaces of dimension $n+1$ and $m+1$, respectively with $n\leq m$. Let $\mathbb{P}^N$ denote the projective space whose points parametrize non-zero linear maps $W\rightarrow V$ up to a scalar multiple. We recall that a point in this spaces is called a collineation from $W$ to $V$.

Observe that a full-rank collineation $p\in \Hom(W,V)$ induces another collineation $\wedge^k W\rightarrow \wedge^k V$, for any $k\le n+1$, by taking the wedge product. Consider the rational map
\stepcounter{thm}
\begin{equation}\label{graph}
\begin{array}{ccc}
\phi:\mathbb{P}(\Hom(W,V))& \dasharrow & \mathbb{P}(\Hom(\bigwedge^2 W,\bigwedge^2 V))\times\dots\times \mathbb{P}(\Hom(\bigwedge^{n+1} W,\bigwedge^{n+1} V))\\
 p & \longmapsto & (\wedge^2p,\dots,\wedge^{n+1}p)
\end{array}
\end{equation}

\begin{Definition}
The \textit{space of complete collineations} $\mathcal{X}(n,m)$ from $W$ to $V$ is defined as the closure of the graph of the rational map (\ref{graph}).
\end{Definition}

In order to understand the closure of the image of $\phi$, we now recall an alternative description of the space $\mathcal{X}(n,m)$.

Given an irreducible, reduced  and non-degenerate variety $X\subset\P^N$, for any $h<N$, we can consider the \emph{$h$-secant variety} of $X$, which is defined as the closure of the union of all $(h-1)$-planes spanned by $h$ general points in $X$. We denote this variety by $\sec_h(X)$. Observe that any point $p\in \mathbb{P}^N = \mathbb{P}(\Hom(W,V))$ can be represented by a $(n+1)\!\!\times \!\!(m+1)$ matrix $Z$, and that the locus of matrices of rank one is the Segre variety $\mathcal{S}\subset\mathbb{P}^N$. More generally, any $p\in \sec_h(S)$ can be represented by a matrix $Z$ which is a linear combination of $h$ matrices of rank one, and conversely. In other words, a point $p$ is contained in $\sec_h(S)$ if and only if the rank of its representing matrix is at most $h$. 

Setting up notation, we will think of the point $p = [z_{0,0}:\cdots:z_{n,m}]\in \mathbb{P}^N$ as the matrix
\stepcounter{thm}
\begin{equation}\label{matrix}
Z = \left(
\begin{array}{ccc}
z_{0,0} & \dots & z_{0,m}\\ 
\vdots & \ddots & \vdots\\ 
z_{n,0} & \dots & z_{n,m}
\end{array}\right)
\end{equation}

We can now describe the space $\mathcal{X}(n,m)$ as a sequence of blow-ups as follows.
\begin{Construction}\label{ccc}
Let us consider the following sequence of blow-ups:
\begin{itemize}
\item[-] $\mathcal{X}(n,m)_1$ is the blow-up of $\mathcal{X}(n,m)_0:=\mathbb{P}^{N}$ along the Segre variety $\mathcal{S}$;
\item[-] $\mathcal{X}(n,m)_2$ is the blow-up of $\mathcal{X}(n,m)_1$ along the strict transform of $\sec_2(\mathcal{S})$;\\
$\vdots$
\item[-] $\mathcal{X}(n,m)_i$ is the blow-up of $\mathcal{X}(n,m)_{i-1}$ along the strict transform of $\sec_i(\mathcal{S})$;\\
$\vdots$
\item[-] $\mathcal{X}(n,m)_n$ is the blow-up of $\mathcal{X}(n,m)_{n-1}$ along the strict transform of $\sec_n(\mathcal{S})$.
\end{itemize}
\end{Construction}

It follows from \cite{Va84} that the variety $\mathcal{X}(n,m)_n$ is isomorphic to $\mathcal{X}(n,m)$, the space of complete collineations from $W$ to $V$. Furthermore, we have that for any $i = 1,\dots,n$ the variety $\mathcal{X}(n,m)_{i}$ is smooth, the strict transform of $\sec_i(\mathcal{S})$ in $\mathcal{X}(n,m)_{i-1}$ is smooth, and the divisor $E_1\cup E_2\cup\dots \cup E_{i-1}$ in $\mathcal{X}(n,m)_{i-1}$ is simple normal crossing. When $n=m$ we will write $\mathcal{X}(n)$ for $\mathcal{X}(n,n)$.

Setting up notation, let $f_i:\mathcal{X}(n,m)_i\rightarrow \mathcal{X}(n,m)_{i-1}$ be the blow-up morphism and $E_i=\mathrm{Exc}(f_i)$ its exceptional divisor. By abusing notation, we will also denote by $E_i$ the strict transform in the subsequent blow-ups. Let $H$ denote the pull-back to $\mathcal{X}(n,m)_i$ of the hyperplane section of $\mathbb{P}^{N}$. We will denote by $f:\mathcal{X}(n,m)\rightarrow\mathbb{P}^N$ the composition of the $f_i$'s.

If $n = m$, then the last blow-up in Construction \ref{ccc} is the blow-up of a Cartier divisor and hence $\mathcal{X}(n)_n\cong\mathcal{X}(n)_{n-1}$. Furthermore, we have a distinguished linear subspace in $\mathbb{P}^N = \mathbb{P}(\Hom(V,V))$. Indeed, the space of symmetric matrices is defined by $\mathbb{P}^{N_{+}} :=\mathbb{P}(\Sym^2V) = \{z_{i,j}=z_{j,i} | \: i\neq j\}$, where $N_{+} = \binom{n+2}{2}-1$. Furthermore, the space $\mathbb{P}^{N_{+}}$ cuts out scheme-theoretically on $\mathcal{S}$ the Veronese variety $\mathcal{V}\subseteq \mathbb{P}^{N_{+}}$; which parametrizes $(n+1)\!\! \times \!\! (n+1)$ symmetric matrices of rank one. More generally, the space $\mathbb{P}^{N_{+}}$ cuts out scheme-theoretically on $\sec_h(\mathcal{S})$ the $h$-secant variety $\sec_h(\mathcal{V})$. 

Observe that a full-rank symmetric matrix $q\in \Sym^2(V)$ induces another symmetric matrix $\wedge^k q$, for any $k\le n+1$, by taking the wedge product. There is a rational map 
\stepcounter{thm}
\begin{equation}\label{ratq}
\begin{array}{ccc}
\rho: \mathbb{P}(\Sym^2V) & \dasharrow & \mathbb{P}(\Sym^2\bigwedge^2 V)\times\dots\times \mathbb{P}(\Sym^2\bigwedge^n V)\\
 q & \longmapsto & (\wedge^2q,\dots,\wedge^{n}q)
\end{array}
\end{equation}

\begin{Definition}
The \textit{space of complete quadrics} in $\P^n$ is defined as the closure of the graph of the map (\ref{ratq}). We denote this space by $\mathcal{Q}(n)$.
\end{Definition}

One can apply Construction \ref{ccc} in the symmetric setting to the spaces $\mathbb{P}^{N_{+}}$. In doing so, we obtain the following blow-up construction for the space of complete quadrics \cite[Theorem 6.3]{Va82}.

\begin{Construction}\label{ccq}
Let us consider the following sequence of blow-ups:
\begin{itemize}
\item[-] $\mathcal{Q}(n)_1$ is the blow-up of $\mathcal{Q}(n)_0:=\mathbb{P}^{N_{+}}$ along the Veronese variety $\mathcal{V}$;
\item[-] $\mathcal{Q}(n)_2$ is the blow-up of $\mathcal{Q}(n)_1$ along the strict transform of $\sec_2(\mathcal{V})$;\\
$\vdots$
\item[-] $\mathcal{Q}(n)_i$ is the blow-up of $\mathcal{Q}(n)_{i-1}$ along the strict transform of $\sec_i(\mathcal{V})$;\\
$\vdots$
\item[-] $\mathcal{Q}(n)_n$ is the blow-up of $\mathcal{Q}(n)_{n-1}$ along the strict transform of $\sec_n(\mathcal{V})$.
\end{itemize}
\end{Construction}

It follows from \cite{Va82} that the variety $\mathcal{Q}(n)_{n-1}\cong\mathcal{Q}(n)_n$ is isomorphic to the space of complete $(n-1)$-dimensional quadrics $\mathcal{Q}(n)$. Furthermore, the variety $\mathcal{Q}(n)_{i}$ is smooth, the strict transform of $\sec_i(\mathcal{V})$ in $\mathcal{Q}(n)_{i-1}$ is smooth, and the divisor $E_1^+\cup E_2^+\cup\dots \cup E_{i-1}^+$ in $\mathcal{Q}(n)_{i-1}$ is simple normal crossing for any $i = 1,\dots,n$. 

Setting up notation, let $f_i^{+}:\mathcal{Q}(n)_i\rightarrow \mathcal{Q}(n)_{i-1}$ be the blow-up morphism and $E_i^+=\mathrm{Exc}(f_i^+)$ its exceptional divisor. We will also denote by $E_i^+$ the strict transforms of this divisor in the subsequent blow-ups, if confusion does not arise. Let $H^{+}$ stand for the pull-back to $\mathcal{Q}(n)_i$ of the hyperplane section of $\mathbb{P}^{N_{+}}$. We will denote by $f^{+}:\mathcal{Q}(n)\rightarrow\mathbb{P}^{N_{+}}$ the composition of the $f_i^{+}$'s. 

\subsection{Birational geometry of the intermediate spaces} 
In this section we will investigate the birational geometry of the intermediate blow-ups $\mathcal{X}(n)_i,\mathcal{Q}(n)_i$ appearing in Constructions \ref{ccc}, \ref{ccq}. We begin by recalling the notion of spherical variety. 
\begin{Definition}
A \textit{spherical variety} is a normal variety $X$ together with an action of a connected reductive affine algebraic group $\mathscr{G}$, a Borel subgroup $\mathscr{B}\subset \mathscr{G}$, and a base point $x_0\in X$ such that the $\mathscr{B}$-orbit of $x_0$ in $X$ is a dense open subset of $X$. 

Let $(X,\mathscr{G},\mathscr{B},x_0)$ be a spherical variety. We distinguish two types of $\mathscr{B}$-invariant prime divisors: a \textit{boundary divisor} of $X$ is a $\mathscr{G}$-invariant prime divisor on $X$, a \textit{color} of $X$ is a $\mathscr{B}$-invariant prime divisor that is not $\mathscr{G}$-invariant. We will denote by $\mathcal{B}(X)$ and $\mathcal{C}(X)$ respectively the set of boundary divisors and colors of $X$.
\end{Definition}

\begin{Notation}
For any $i = 1,\dots,n+1$ let us denote by $D_i$ the strict transform of the divisor in $\mathbb{P}^N$ given by 
\stepcounter{thm}
\begin{equation}\label{eqcol}
\det\left(
\begin{array}{ccc}
z_{n-i+1,m-i+1} & \dots & z_{n-i+1,m}\\ 
\vdots & \ddots & \vdots\\ 
z_{n,m-i+1} & \dots & z_{n,m}
\end{array}\right)=0
\end{equation}
in the intermediate spaces of complete collineations. Similarly, we will denote by $D_i^{+}$ the analogous divisors in the intermediate spaces of complete quadrics.
\end{Notation}

\begin{Lemma}\label{cb}
Let $\mathcal{X}(n,m)_i$ be the intermediate space appearing at the step $1\leq i\leq n-1$ of Construction \ref{ccc}. Then $\mathcal{X}(n,m)_i$ is spherical. Furthermore, $\Pic(\mathcal{X}(n,m)_i) = \mathbb{Z}[H,E_1,\dots,E_i]$ and 
$$
\mathcal{B}(\mathcal{X}(n,m)_i) = 
\left\lbrace\begin{array}{ll}
\{E_1,\dots,E_{i}\} & \text{if n $<$ m}\\ 
\{E_1,\dots,E_{i},D_{n+1}\} & \text{if n $=$ m}
\end{array}\right.\quad
\mathcal{C}(\mathcal{X}(n,m)_i) = 
\left\lbrace\begin{array}{ll}
\{D_1,\dots,D_{n+1}\} & \text{if n $<$ m}\\ 
\{D_1,\dots,D_{n}\} & \text{if n $=$ m}
\end{array}\right.
$$
Finally, for the intermediate spaces $\mathcal{Q}(n)_i$ in Construction \ref{ccq} we have
$$\mathcal{B}(\mathcal{Q}(n)_i) = \{E_1^{+},\dots,E_i^{+},D_{n+1}^{+}\},\quad \mathcal{C}(\mathcal{Q}(n)_i) = \{D_1^{+},\dots,D_{n}^{+}\}.$$  
\end{Lemma}
\begin{proof}
Note that $\mathcal{X}(n,m)_i$ is spherical, even though it is not wonderful, with respect to the action of $\mathscr{G} = SL(n+1)\times SL(m+1)$, with the Borel subgroup $\mathscr{B}$ given by pair of upper triangular matrices. Indeed, by \cite[Theorem 1]{Va84} $\mathcal{X}(n,m)$ is a wonderful variety and $\mathscr{B}$ acts on it with a dense open orbit, and since the actions of $\mathscr{B}$ on $\mathcal{X}(n,m)$ and $\mathcal{X}(n,m)_i$ are compatible with the blow-up maps in Construction \ref{ccc} we conclude that $\mathscr{B}$ acts on $\mathcal{X}(n,m)_i$ with a dense open orbit as well. For the definition and the basic properties of wonderful varieties we refer to \cite[Section 4.5.5]{ADHL15}.  

Since the actions of $\mathscr{G}$ on $\mathcal{X}(n,m)$ and $\mathcal{X}(n,m)_i$ are compatible with the blow-up morphism $\mathcal{X}(n,m)\rightarrow \mathcal{X}(n,m)_i$, the intermediate space $\mathcal{X}(n,m)_i$ has at most the same number of boundary divisors and colors as $\mathcal{X}(n,m)$. Now, to compute $\mathcal{B}(\mathcal{X}(n,m)_i)$ and $\mathcal{C}(\mathcal{X}(n,m)_i)$ it is enough to note that the ones listed in the statement are clearly boundary divisors and colors, and to apply \cite[Proposition 3.6]{Ma18a}.
\end{proof}

\begin{Proposition}\label{theff}
Let $n< m$. Then $D_1\sim H$ and
\stepcounter{thm}
\begin{equation}\label{Dk}
D_k \sim kH-\sum_{h=1}^{k-1}(k-h)E_h
\end{equation}
for $k = 2,\dots,n+1$. Furthermore, $\Eff(\mathcal{X}(n,m)_i) = \left\langle E_1,\dots,E_i,D_{n+1}\right\rangle$ and $\Nef(\mathcal{X}(n,m)_i) = \left\langle D_1,\dots,D_{i+1}\right\rangle$.

If $n = m$ we have $D_1\sim H$, $D_1^{+}\sim H^{+}$ and
\stepcounter{thm}
\begin{equation}\label{Dkq}
D_k \sim kH-\sum_{h=1}^{k-1}(k-h)E_h, \quad D_k^{+} \sim kH^{+}-\sum_{h=1}^{k-1}(k-h)E_h^{+}
\end{equation}
for $k = 2,\dots,n$, and for the boundary divisors $E_n$ and $E_n^{+}$ we have
\stepcounter{thm}
\begin{equation}\label{lastsec}
E_{n}\sim (n+1)H-\sum_{h=1}^{n-1}(n-h+1)E_h, \quad E_{n}^{+}\sim (n+1)H^{+}-\sum_{h=1}^{n-1}(n-h+1)E_h^{+}.
\end{equation}
Furthermore, $\Eff(\mathcal{X}(n)_i) = \left\langle E_1,\dots,E_{i},D_{n+1}\right\rangle$, $\Nef(\mathcal{X}(n)_i) = \left\langle D_1,\dots,D_{i+1}\right\rangle$ and similarly $\Eff(\mathcal{Q}(n)_i) = \left\langle E_1^{+},\dots,E_{i}^{+},D_{n+1}^{+}\right\rangle$, $\Nef(\mathcal{Q}(n)_i) = \left\langle D_1^{+},\dots,D_{i+1}^{+}\right\rangle$ for $i = 1,\dots,n-1$. 
\end{Proposition}
\begin{proof}
Consider the case $n < m$, a similar argument works for the case $n = m$. Let $f:\mathcal{X}(n,m)\rightarrow\mathbb{P}^N$ be the blow-up morphism in Construction \ref{ccc}. We have 
$$D_k = \deg(f_{*}D_k)H-\sum_{h=1}^{i}\mult_{\sec_h(\mathcal{S})}(f_{*}D_k)E_h$$
Now, since $f_{*}D_k$ is the zero locus of a $k\times k$ minor of the matrix $Z$ then $\deg(f_{*}D_k)=k$. Furthermore, \cite[Lemma 3.10]{Ma18a} yields $\mult_{\sec_h(\mathcal{S})}(f_{*}D_k) = k-h$ if $h = 1,\dots k-1$ and $\mult_{\sec_h(\mathcal{S})}(f_{*}D_k) = 0$ if $h\geq k$, and hence we get (\ref{Dk}).

By \cite[Proposition 4.5.4.4]{ADHL15} and \cite[Section 2.6]{Br89} the generators of the effective and the nef cones are respectively generated by the boundary divisors and the colors. Lemma \ref{cb} then concludes the proof.
\end{proof}

\subsubsection{Generators of $\Cox(\mathcal{X}(n,m)_i)$}
Let $I = \{i_0,\dots, i_k\}$, $J = \{j_0,\dots, j_k\}$ be two ordered sets of indexes with $0\leq i_0\leq\dots\leq i_k\leq n$ and $0\leq j_0\leq\dots\leq j_k\leq m$, and denote by $Z_{I,J}$ the $(k+1)\times (k+1)$ minor of $Z$ built with the rows indexed by $I$ and the columns indexed by $J$. Similarly, let $Z_{I,J}^{+}$ be the $(k+1)\times (k+1)$ minor of the symmetrization of $Z$ built with the rows indexed by $I$ and the columns indexed by $J$.  

\begin{Proposition}\label{gen}
Let $T_{I,J}$ be the canonical section associated to the strict transform of the hypersurface $\{\det(Z_{I,J})=0\}\subset\mathbb{P}^N$, and let $S_j$ be the canonical section associated to the exceptional divisor $E_j$ in Construction \ref{ccc}. Then $\Cox(\mathcal{X}(n,m))_i$ is generated by the $T_{I,J}$ for $1\leq |I|,|J|\leq n+1$ and the $S_j$ for $j=1,\dots,i$. 

Now, consider $\mathcal{Q}(n)_i$. Let $T_{I,J}^{+}$ be the canonical section associated to the strict transform of the hypersurface $\{\det(Z_{I,J}^{+})=0\}\subset\mathbb{P}^{N_{+}}$, and let $S_j^{+}$ be the canonical section associated to the exceptional divisor $E_j^{+}$ in Construction \ref{ccq}. Then $\Cox(\mathcal{Q}(n)_i)$ is generated by the $T_{I,J}^{+}$ for $1\leq |I|,|J|\leq n+1$ and the $S_j^{+}$ for $j=1,\dots,i$. 
\end{Proposition}
\begin{proof}
By \cite[Theorem 4.5.4.6]{ADHL15} if $\mathscr{G}$ is a semi-simple and simply connected algebraic group and $(X,\mathscr{G},\mathscr{B},x_0)$ is a spherical variety with boundary divisors $E_1,\dots,E_r$ and colors $D_1,\dots,D_s$, then $\Cox(X)$ is generated as a $K$-algebra by the canonical sections of the $E_i$'s and the finite dimensional vector subspaces $\lin_{K}(\mathscr{G}\cdot D_i)\subseteq \Cox(X)$ for $1\leq i\leq s$.

Consider the case $n < m$. For $\mathcal{X}(n)_i$ and $\mathcal{Q}(n)_i$ an analogous argument works. By Lemma \ref{cb} we have that $\mathcal{B}(\mathcal{X}(n,m)_i) = \{E_1,\dots,E_i\}$ and $\mathcal{C}(\mathcal{X}(n,m)_i) = \{D_1,\dots,D_{n+1}\}$. Recall that for any $k=0,\dots, n$ the divisor $D_{k+1}$ is the strict transform of the hypersurface in $\mathbb{P}^N$ defined by the determinant of the $(k+1)\times (k+1)$ most right down minor of the matrix $Z$. Let us denote by $Z_{I,J}^{rd}$ such minor. 

Note that $\det(Z_{I,J}^{rd})\in I(\sec_{k}(\mathcal{S}))$. Therefore, considering the action of $\mathscr{G} =  SL(n+1)\times SL(m+1)$ we have that $g\cdot \det(Z_{I,J}^{rd})\in I(\sec_{k}(\mathcal{S}))$ for any $g\in \mathscr{G}$, and hence $\lin_{K}(\mathscr{G}\cdot \det(Z_{I,J}^{rd})) \subseteq I(\sec_{k}(\mathcal{S}))$. To conclude it is enough to recall that $I(\sec_{k}(\mathcal{S}))$ is generated by the $(k+1)\times (k+1)$ minors of $Z$.
\end{proof}

\subsection{The Mori chamber decomposition of the intermediate space of Picard rank three}
In this section, using the techniques introduced in \cite{Ma18a, Ma18b}, we will study the Mori chamber and stable base locus decomposition for the spaces $\mathcal{X}(n)_3$ in Construction \ref{ccc}. 

\begin{Notation}
We will denote by $\left\langle v_1,\dots,v_s\right\rangle$ the cone in $\mathbb{R}^n$ generated by the vectors $v_1,\dots,v_s\in\mathbb{R}^n$. Given two vectors $v_i,v_j$ we set $(v_i,v_j] := \left\langle v_i,v_j\right\rangle\setminus \mathbb{R}_{>0} v_i$ and $(v_i,v_j) := \left\langle v_i,v_j\right\rangle\setminus (\mathbb{R}_{>0} v_i\cup \mathbb{R}_{>0} v_j)$.
\end{Notation}

\begin{Lemma}\label{lemma_mov}
Consider the natural inclusion $i:\mathcal{Q}(n)_i\rightarrow\mathcal{X}(n)_i$. Then $i^{*}\Mov(\mathcal{X}(n)_i)=\Mov(\mathcal{Q}(n)_i)$.
\end{Lemma}
\begin{proof}
By Proposition \ref{gen} $\Cox(\mathcal{Q}(n)_i)$ is generated by the pull-backs of the generators of $\Cox(\mathcal{X}(n)_i)$. Hence the statement follows from \cite[Proposition 3.3.2.3]{ADHL15}.   
\end{proof}

\begin{Proposition}\label{mov}
The movable cone of $\mathcal{X}(n)_3$ is generated by $D_1\sim H, D_2\sim 2H-E_1, D_n\sim nH+(1-n)E_1+(2-n)E_2, P_n\sim n(n-1)H+n(2-n)E_1+(n-1)(2-n)E_2$. 

Similarly, $\Mov(\mathcal{Q}(n)_3)$ is generated by $D_1^{+}\sim H^{+}, D_2^{+}\sim 2H^{+}-E_1^{+}, D_n^{+}\sim nH^{+}+(1-n)E_1^{+}+(2-n)E_2^{+}, P_n^{+}\sim n(n-1)H^{+}+n(2-n)E_1^{+}+(n-1)(2-n)E_2^{+}$. 
\end{Proposition}
\begin{proof}
We use \cite[Proposition 3.3.2.3]{ADHL15} which gives a way of computing the movable cone starting from the generators of the Cox ring. More specifically, in order to compute the movable cone it is enough to intersect all the cones generated by the sets of vectors obtained by removing just one element from the set of generators of the Cox ring.

In our specific case, using the generators of $\Cox(\mathcal{X}(n)_3)$ in Proposition \ref{gen}, it is immediate to see that this procedure give us the vectors $(1,0,0),(2,-1,0),(n,1-n,2-n)$ and the vector $(n(n-1),n(2-n),(n-1)(2-n))$ that is the intersection of the plane generated by $(1,0,0),(n+1,-n,-n+1)$ and the plane generated by $(0,1,0),(n,1-n,2-n)$. Finally, the statement on $\Mov(\mathcal{Q}(n)_3)$ follows from Lemma \ref{lemma_mov}.
\end{proof}

\begin{thm}\label{MCD_main}
Let $f(n)$ be the number of chambers in the Mori chamber decomposition of $\mathcal{X}(n)_{3}$. Then $f(n+1) = f(n)+n+1$. Furthermore, the decomposition can be described by the following $2$-dimensional cross-section of $\Eff(\mathcal{X}(n)_3)$
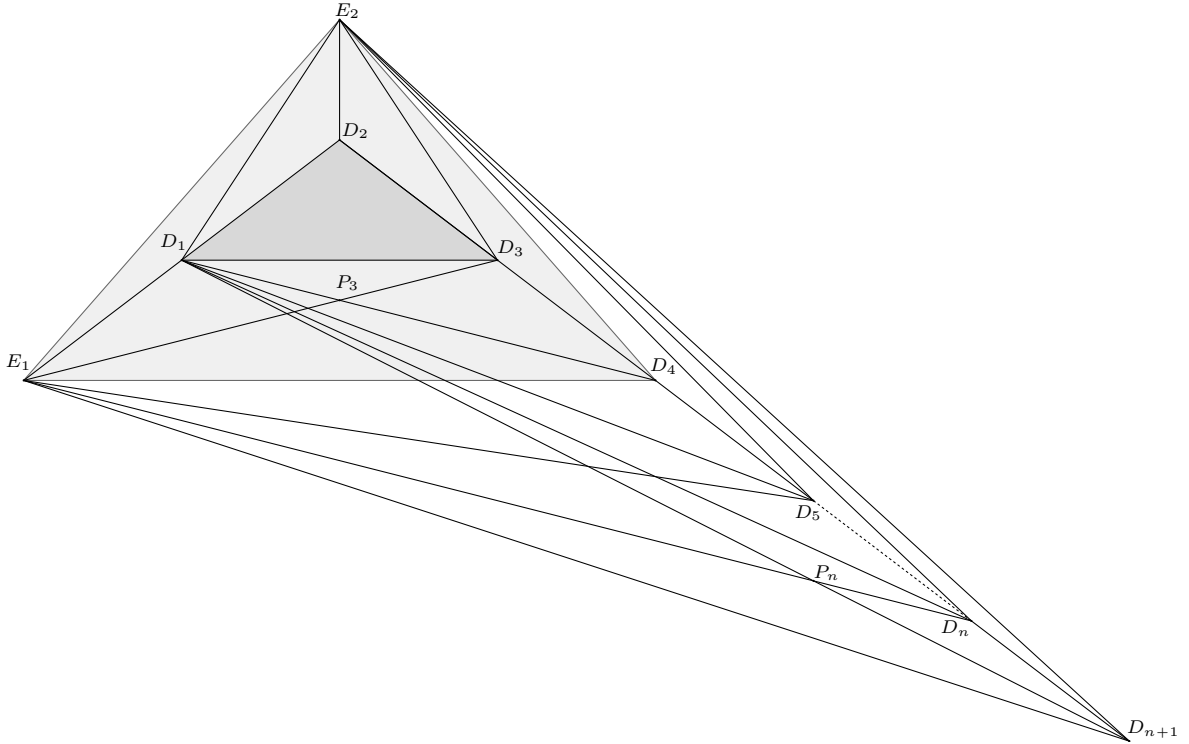
$$
\definecolor{wwwwww}{rgb}{0.4,0.4,0.4}
\begin{tikzpicture}[xscale=0.35,yscale=0.40][line cap=round,line join=round,>=triangle 45,x=1cm,y=1cm]
\clip(-14.314487632508836,-15.99593495934956) rectangle (42.233215547703175,9.0);
\fill[line width=0.4pt,color=wwwwww,fill=wwwwww,fill opacity=0.1] (-12,-4) -- (12,-4) -- (0,8) -- cycle;
\fill[line width=0.4pt,fill=black,fill opacity=0.1] (0,4) -- (-6,0) -- (6,0) -- cycle;
\draw [line width=0.4pt,color=wwwwww] (-12,-4)-- (12,-4);
\draw [line width=0.4pt,color=wwwwww] (12,-4)-- (0,8);
\draw [line width=0.4pt,color=wwwwww] (0,8)-- (-12,-4);
\draw [line width=0.4pt] (0,4)-- (-6,0);
\draw [line width=0.4pt] (-6,0)-- (6,0);
\draw [line width=0.4pt] (6,0)-- (0,4);
\draw [line width=0.4pt] (0,4)-- (18,-8);
\draw [line width=0.4pt,dash pattern=on 1pt off 1pt] (18,-8)-- (24,-12);
\draw [line width=0.4pt] (24,-12)-- (30,-16);
\draw [line width=0.4pt] (-6,0)-- (-12,-4);
\draw [line width=0.4pt] (-6,0)-- (0,8);
\draw [line width=0.4pt] (6,0)-- (0,8);
\draw [line width=0.4pt] (6,0)-- (-12,-4);
\draw [line width=0.4pt] (-6,0)-- (12,-4);
\draw [line width=0.4pt] (0,4)-- (0,8);
\draw [line width=0.4pt] (0,8)-- (18,-8);
\draw [line width=0.4pt] (18,-8)-- (-12,-4);
\draw [line width=0.4pt] (18,-8)-- (-6,0);
\draw [line width=0.4pt] (0,8)-- (24,-12);
\draw [line width=0.4pt] (24,-12)-- (-12,-4);
\draw [line width=0.4pt] (-6,0)-- (24,-12);
\draw [line width=0.4pt] (0,8)-- (30,-16);
\draw [line width=0.4pt] (30,-16)-- (-12,-4);
\draw [line width=0.4pt] (30,-16)-- (-6,0);
\begin{scriptsize}
\draw [fill=black] (-12,-4) circle (0.5pt);
\draw[color=black] (-12.2,-3.4) node {$E_1$};
\draw [fill=black] (12,-4) circle (0.5pt);
\draw[color=black] (12.303886925795048,-3.5254065040650344) node {$D_4$};
\draw [fill=black] (0,8) circle (0.5pt);
\draw[color=black] (0.2932862190812683,8.3) node {$E_2$};
\draw [fill=black] (0,4) circle (0.5pt);
\draw[color=black] (0.6,4.3) node {$D_2$};
\draw [fill=black] (-6,0) circle (0.5pt);
\draw[color=black] (-6.3,0.6) node {$D_1$};
\draw [fill=black] (6,0) circle (0.5pt);
\draw[color=black] (6.5,0.4) node {$D_3$};
\draw [fill=black] (18,-8) circle (0.5pt);
\draw[color=black] (17.8,-8.4) node {$D_5$};
\draw [fill=black] (24,-12) circle (0.5pt);
\draw[color=black] (23.4,-12.25) node {$D_n$};
\draw [fill=black] (30,-16) circle (0.5pt);
\draw[color=black] (30.9,-15.6) node {$D_{n+1}$};
\draw [fill=black] (0,-1.3333333333333333) circle (0.5pt);
\draw[color=black] (0.2932862190812683,-0.75) node {$P_3$};
\draw [fill=black] (18,-10.666666666666666) circle (0.5pt);
\draw[color=black] (18.5,-10.4) node {$P_n$};
\end{scriptsize}
\end{tikzpicture}
$$
Furthermore, the same statements hold, by replacing the relevant divisors with their pull-backs via the embedding $i:\mathcal{Q}(n)_3\rightarrow \mathcal{X}(n)_3$, for the intermediate space of quadrics $\mathcal{Q}(n)_3$. 
\end{thm}
\begin{proof} 
By \cite[Theorem 6.11]{Ma18a} the statement holds for $n = 3$. Consider the case $n = 4$. Then the Cox ring of $\mathcal{X}(4)_3$ has additional generators belonging to the class of $D_5$.

Now, let $\mathcal{T}_{\mathcal{X}(4)_3}$ be a projective toric variety as in Remark \ref{toric}. Then there is an embedding $i:\mathcal{X}(4)_3\rightarrow\mathcal{T}_{\mathcal{X}(4)_3}$ such that $i^{*}:\Pic(\mathcal{T}_{\mathcal{X}(4)_3})\rightarrow \Pic(\mathcal{X}(4)_3)$ is an isomorphism inducing an isomorphism $\Eff(\mathcal{T}_{\mathcal{X}(4)_3})\rightarrow \Eff(\mathcal{X}(4)_3)$. 

Since $\mathcal{T}_{\mathcal{X}(4)_3}$ is toric, the Mori chamber decomposition of $\Eff(\mathcal{T}_{\mathcal{X}(4)_3})$ can be computed by means of the Gelfand-Kapranov-Zelevinsky
decomposition \cite[Section 2.2.2]{ADHL15}. Roughly speaking such decomposition is given by all the convex cones we can construct using all the generators of $\Cox(\mathcal{T}_{\mathcal{X}(4)_3})$. Starting from the decomposition for $\mathcal{T}_{\mathcal{X}(4)_3}$ we are allowed to add just two more segments, namely $[E_2,D_5], [E_1,D_5]$ which introduce a new chamber each, and $[D_1,D_5]$ which gives rise two three new chambers by dividing $\left\langle D_1,D_3,E_1\right\rangle, \left\langle E_1,D_3,D_4\right\rangle, \left\langle E_1,D_4,D_5\right\rangle$. Note that we get exactly $f(4) = f(3)+4+1 = 9+4+1 = 14$ chambers. On the other hand, a priori this is just a refinement of the Mori chamber decomposition of $\Eff(\mathcal{X}(4)_3)$. We will prove that it is indeed the Mori chamber decomposition. 

Note that by considering a general $\mathbb{P}^3\times \mathbb{P}^3\subset\mathbb{P}^4\times \mathbb{P}^4$ we get a natural embedding $j:\mathcal{X}(3)_3\rightarrow \mathcal{X}(4)_3$. Furthermore, this embedding preserves the generators of the Picard groups in Lemma \ref{cb} and the generators of the Cox rings in Proposition \ref{gen}. Therefore, all the chambers of $\MCD(\mathcal{X}(3)_3)$ appear in $\MCD(\mathcal{X}(4)_3)$. Furthermore, the segments $[E_2,D_5], [E_1,D_5]$ must appear in $\MCD(\mathcal{X}(4)_3)$ simply because by Proposition \ref{theff} they define two external walls of $\Eff(\mathcal{X}(4)_3)$, and the segment $[D_1,D_5]$ must appear as well since it is needed to intercept the vector $P_4$ which by Proposition \ref{mov} is an extremal ray of $\Mov(\mathcal{X}(4)_3)$.

Now, assuming that the statement holds for $\mathcal{X}(n-1)_3$ we can prove it for $\mathcal{X}(n)_3$ arguing exactly as we did in order to pass from $\mathcal{X}(3)_3$ to $\mathcal{X}(4)_3$. Indeed, arguing exactly as in the previous part of the proof we see that all the chambers of $\MCD(\mathcal{X}(n-1)_3)$ must appear in $\MCD(\mathcal{X}(n)_3)$. Moreover, we have to add the segments $[E_1,D_{n+1}],[E_2,D_{n+1}]$ which by Proposition \ref{theff} give external walls of $\Eff(\mathcal{X}(n)_3)$, and $[D_1,D_{n+1}]$ needed to intercept the vectors $P_n$ which by Proposition \ref{mov} is an extremal ray of $\Mov(\mathcal{X}(n)_3)$.

Finally, note that $n-2$ of the chambers in $\mathcal{X}(n-1)_3$ are subdivided in two chambers each by $[D_1,D_{n+1}]$, the new cone $\left\langle E_1,D_n,D_{n+1}\right\rangle$ also is subdivided in two chambers each by $[D_1,D_{n+1}]$, and we must add also the chamber $\left\langle E_2,D_n,D_{n+1}\right\rangle$. Hence the number of chambers of $\Eff(\mathcal{X}(n)_3)$ is given by $f(n) = f(n-1)+n-2+2+1 = f(n-1)+n+1$.
\end{proof}

\begin{Remark}\label{rem_3}
By \cite[Theorem 6.11]{Ma18a} $\mathcal{X}(3)_3$ has nine Mori chambers but just eight stable base locus chambers. Indeed, within the non-convex chamber determined by the divisors $D_1,D_2,D_3,E_2$ the stable base locus is $E_2$. Note that this last fact holds true more generally for the space $\mathcal{X}(n)_3$ for $n\geq 3$. Similarly, this holds also for the spaces of quadrics $\mathcal{Q}(n)_3$.   
\end{Remark}

\begin{Corollary}\label{CQ3}
The varieties $\mathcal{X}(n)_i$ and $\mathcal{Q}(n)_i$ are Lefschetz divisorially equivalent, and moreover $\mathcal{X}(n)_3$ and $\mathcal{Q}(n)_3$ are birational twins for any $n\geq 1$. Furthermore, $\mathcal{X}(3)_3$ and $\mathcal{Q}(3)_3$ are strong birational twins.
\end{Corollary}
\begin{proof}
The first statement follows from Proposition \ref{theff} and Lemma \ref{lemma_mov}. The second statement follows then from Theorem \ref{MCD_main}. Finally, the third statement is a consequence of Remark \ref{rem_3}.   
\end{proof}

\section{The flip of the spaces of complete quadric surfaces}\label{sec_flip}

Theorem \ref{MCD_main} and Corollary \ref{CQ3} imply that each of the spaces $\mathcal{X}(3),\mathcal{Q}(3)$ of complete collineations of $\P^3$ and complete quadric surfaces, admits a single flip. Moreover, such flips are strongly related to each other. Indeed, by computing the Mori chamber decomposition in both cases, we observe that $i^{*}\mathrm{MCD}(\mathcal{X}(3)) = \mathrm{MCD}(\mathcal{Q}(3))$, thus these two spaces are birational twins and the unique flip on $\mathcal{X}(3)$ induces that of $\mathcal{Q}(3)$.

In this section we exhibit the flip of $\mathcal{Q}(3)$, denoted by $\mathcal{Q}(3)^{+}$, the space of complete quadric surfaces following \cite{Ce15}. Towards the end of the section, we conjecture the geometry of the flip of $\mathcal{X}(3)$. We construct $\mathcal{Q}(3)^{+}$ by analyzing a $\ZZ/2$-action on the following Hilbert scheme.

\begin{Proposition}
Let $\mathbf{Hilb}=\mathbf{Hilb}^{2x+1}(\mathbb{G}(1,3))$ denote the Hilbert scheme parametrizing subschemes of $\mathbb{G}(1,3)\subset \P^5$ whose Hilbert polynomial is $P(x)=2x+1$. This space is isomorphic to the following blow-up 
$$\mathbf{Hilb}\cong Bl_{\mathbb{OG}}\mathbb{G}(2,5)$$ 
where $\mathbb{O G}\subset \GG(2,5)$ denotes the orthogonal Grassmannian inside the Grassmannian of $2$-planes in $\P^5$.
\end{Proposition}
\begin{proof}
Observe that a generic smooth curve with Hilbert polynomial $P(x)=2x+1$ in $\P^5$ is a plane conic $C$. Thus, its ideal $I_C\subset k[p_0,...,p_5]$ is generated by a quadric $F$ and three independent linear forms $L_1,L_2,L_3$. Since $C\subset \mathbb{G}=\mathbb{G}(1,3)$, the equation $F$ is the quadric corresponding to the Grassmannian $\GG \subset \P^5$ under the Pl\"{u}cker embedding.
This description gives rise to a rational map $$f:\GG(2,5)\dashrightarrow \mathbf{Hilb} $$ by assigning the $2$-plane $P$ defined by the independent linear forms $(L_1,L_2,L_3)$ to the ideal $\langle L_1,L_2,L_3\rangle+\langle F\rangle \subset k[p_0,...,p_5]$. 
Observe that the exceptional locus of $f$ consists of planes in $\P^5$ such that there is a containment of ideals $\langle F\rangle \subset \langle L_1,L_2,L_3\rangle$, that is planes $P$ which are contained in the quadric $\GG \subset \P^5$. The locus parametrizing such planes is exactly the orthogonal Grassmannian $\mathbb{OG}$. Now, we resolve the rational map $f$,
  \[
  \begin{tikzpicture}[xscale=2.8,yscale=-1.4]
    \node (A0_0) at (0, 0) {$Bl_{\mathbb{OG}}\mathbb{G}(2,5)$};
    \node (A1_0) at (0, 1) {$\mathbb{G}(2,5)$};
    \node (A1_1) at (1, 1) {$\textbf{Hilb}$};
    \path (A1_0) edge [->,dashed]node [auto] {$\scriptstyle{f}$} (A1_1);
    \path (A0_0) edge [->,swap]node [auto] {$\scriptstyle{\pi}$} (A1_0);
    \path (A0_0) edge [->]node [auto] {$\scriptstyle{\widetilde{f}}$} (A1_1);
  \end{tikzpicture}
  \]
The morphism $\tilde f$ is an isomorphism. Indeed, the rational map $f$ is birational as it has an inverse morphism $g:\mathbf{Hilb}\rightarrow \GG(2,5)$ defined as follows: let $[C]\in \mathbf{Hilb}$ be a point, then the ideal $I(C)=(F)+ (\mbox{plane})\overset{g}{\mapsto}(\mbox{plane})\in \GG(2,5)$. It is clear that $f \circ g=Id$, hence $f$, and consequently $\tilde f$, is birational. Furthermore, $\tilde f$ is a bijection. Indeed, since the exceptional divisor $E\subset Bl_{\mathbb{OG}}\GG(2,5)$ is a $\P^5$-bundle over $\mathbb{OG}$, whose points can be thought of as pairs $(P,C)$, where $P\subset \P^5$ is a $2$-plane and $C\subset P$ is a plane conic. Thus,  Zariski's Main Theorem implies that $\tilde f$ is an isomorphism.
\end{proof}
\begin{Corollary}
Let $\mathbf{Hilb}$ be as above, then $\mathrm{Pic}(\mathbf{Hilb})\cong \langle H^{+},E_2^{+},E_{1,1}^{+}\rangle$ where $H^{+}$ is the pullback of the unique generator of the group $A^1(\GG(2,5))$ and the $E^{+}$'s are the exceptional divisors of the blow-up.
\end{Corollary}
\begin{proof}
The orthogonal Grassmannian $\mathbb{OG}$ has two components, hence the result follows.
\end{proof}

If the base field $K$ is algebraically closed, then for a given smooth quadric $Q\subset \P^3$, the Fano variety of lines $F_1(Q)\subset \GG(1,3)$ consists of two smooth conics. We get a $\ZZ/2$-action on $\mathbf{Hilb}^{2x+1}(\GG(1,3))$ by exchanging such conics. 

\begin{Lemma}\label{ACT}
There is a nontrivial globally defined $\ZZ/2$-action on $\mathbf{Hilb}^{2x+1}(\GG(1,3))$. 
\end{Lemma}
\begin{proof} Let $Q\subset \P^3$ be a smooth quadric hypersurface. The Fano variety of lines $F_1(Q)$ is the zero locus of a section of the following bundle, 
\[
\begin{tikzpicture}[xscale=1.5,yscale=-1.4]
  \node (A0_0) at (0, 0) {$\Sym^2(S^{*})$};
  \node (A1_0) at (0, 1) {$\mathbb{G}(1,3)$};
  \path (A0_0) edge [->]node [auto] {$\scriptstyle{\pi}$} (A1_0);
  \path (A1_0) edge [->,bend right=35]node [auto] {$\scriptstyle{Q_{|L}}$} (A0_0);
\end{tikzpicture}
\]where $S^*$ is the dual of the tautological bundle $S$ over $\mathbb{G}(1,3)$.
A smooth conic in $\P^5$ determines uniquely a $2$-plane, thus in the Pl\"{u}cker embedding $\GG(1,3)\subset \P^5$, we have that
\begin{itemize}
\item[-] $F_1(Q)$ determines two $2$-planes if $\rank(Q)$ is either $2$ or $4$,
\item[-] $F_1(Q)$ determines a single $2$-plane if $\rank(Q)$ is either $1$ or $3$.
\end{itemize}
Exchanging such planes gives rise to a $\ZZ/2$-action on 
$\GG(2,5)$, the Grassmannian of $2$-planes in $\P^5$. The second condition above, says that such a 
$\ZZ/2$-action on $\GG(2,5)$ preserves the orthogonal Grassmannian $\mathbb{OG}$, hence inducing a $\ZZ/2$-action on the blow-up $\mathbf{Hilb}^{2x+1}(\GG(1,3))$.
\end{proof}

Observe that there is an $SL_4(\mathbb{C})$-action on $\mathbf{Hilb}$ induced by the action of $SL_4$ on $\P^3$. This action stratifies $\mathbf{Hilb}$ into $SL_4$-orbits compatible with the exceptional divisors $E_2^+,E_{1,1}^+$.
Notice that $\ZZ/2$ acts trivially on $SL_4$-orbits of codimension $2$. In codimension $1$, we have that $\ZZ/2$ acts as the identity on the exceptional divisors $E_2^{+}$ and $E_{1,1}^{+}$. Consider the quotient  
$$\mathcal{Q}(3)^{+}:=\mathbf{Hilb}/(\ZZ/2)$$ 
where the $\ZZ/2$-action is defined in the previous Lemma \ref{ACT}.The main result of this section is the following.

\begin{thm}\label{flip_th}
There is a flip $f:\mathcal{Q}(3)\dasharrow \mathcal{Q}(3)^+ $ over the Chow variety $\mathbf{Chow}^{2x+1}(\GG(1,3))$, and the flipping locus in $\mathcal{Q}(3)$ is the intersection the the divisors $E_1\cap E_3$.
\end{thm}

The previous result can summarized in the following diagram
  \[
  \begin{tikzpicture}[xscale=2.8,yscale=-1.4]
    \node (A0_1) at (1, 0) {$\mathcal{Q}(3)$};
    \node (A0_3) at (3, 0) {$\mathcal{Q}(3)^{+}$};
    \node (A1_0) at (0, 1) {$\mathbb{P}^9$};
    \node (A1_2) at (2, 1) {$\mathcal{C}/(\mathbb{Z}/2)$};
    \node (A1_4) at (4, 1) {$\mathbb{G}(2,5)/(\mathbb{Z}/2)$};
    \path (A0_3) edge [->]node [auto] {$\scriptstyle{\phi_2}$} (A1_2);
    \path (A0_1) edge [->,swap]node [auto] {$\scriptstyle{\phi_1}$} (A1_2);
    \path (A0_1) edge [->,dashed]node [auto] {$\scriptstyle{flip}$} (A0_3);
    \path (A0_3) edge [->]node [auto] {$\scriptstyle{}$} (A1_4);
    \path (A0_1) edge [->]node [auto] {$\scriptstyle{}$} (A1_0);
  \end{tikzpicture}
  \]
where $\phi_1$ and $\phi_2$ are small contractions and the other maps, except for the flip, are all divisorial contractions described in previous sections. 

In order to show that $\mathcal{Q}(3)$ and $\mathcal{Q}(3)^+$ are related through a flip, we lift the maps so we avoid the action of $\mathbb{Z}/2$. In this scenario, we need to describe the flip for the Hilbert scheme $\mathbf{Hilb}^{2x+1}(\mathbb{G}(1,3))$. In doing so, we are lead to consider the Kontsevich moduli space $\overline{\mathcal{M}}_{0,0}(\GG,2)$ of degree two stable maps into the Grassmannian $\GG=\GG(1,3)$.

\begin{Lemma}\label{2COVER}
There is a nontrivial globally defined $\ZZ/2$-action on the Kontsevich space $\overline{\mathcal{M}}_{0,0}(\GG(1,3),2)$. 
\end{Lemma}
\begin{proof}
We have a generic $2$ to $1$ morphism from the Kontsevich moduli space $\overline{\mathcal{M}}=\overline{\mathcal{M}}_{0,0}(\GG(1,3),2)=\{(C,C^*)\}$ to the space $\mathcal{Q}(3)$ of complete quadric surfaces defined as follows 
$$(C,C^*)\mapsto \left(\bigcup_{L\in C}L,C^*\right)$$
where the notation $(S,C^*)$ means a surface $S$, and a curve $C^*$ as its marking. This map is $2$ to $1$ over the open subset parametrizing smooth quadric surfaces as well as over the divisor of complete quadrics of rank two. Indeed, if $\bigcup_{L\in C} L$ sweeps out a smooth quadric $S$, then $L$ is a ruling of $S$. The other ruling induces another conic $C'$ which gets mapped to $S$. The situation is similar over the locus of complete quadrics of rank two. Note that this map is $1$ to $1$ outside two such regions. We now define the $\ZZ/2$-action on $\overline{\mathcal{M}}$ by identifying the two curves $C$ and $C'$.
\end{proof}

\begin{Corollary}
The quotient of $\overline{\mathcal{M}}_{0,0}(\GG,2)$ by the $\ZZ/2$-action is isomorphic to  $\mathcal{Q}(3)$. In particular, the quotient is smooth.
\end{Corollary}
\begin{proof}
Let $Z$ denote the quotient of $\overline{\mathcal{M}}$ by the $\ZZ/2$-action defined above. Observe that $X_3$ and $Z$ are birational and there is a bijection between them. Zariski's Main Theorem implies now the corollary.
\end{proof} 

It follows from \cite{CC10} that $\overline{\mathcal{M}}_{0,0}(\GG,2)$ and $\mathbf{Hilb}^{2x+1}(\mathbb{G}(1,3))$ are related through a flip over the Chow variety. Furthermore, the maps in such a flip diagram between these two spaces are $\mathbb{Z}/2$-equivariant \cite[Section 5.2]{Ce15}. The result in Theorem \ref{flip_th} now follows from the previous lemmas.

Following the construction of the space $\mathcal{Q}(3)^+$ above, we conjecture a modular interpretation for the flip of the space of complete collineations $\mathcal{X}(3)$. Let $A\in \mathcal{X}(3)$ be a generic complete collineation of $\P^3$. Then, we associate to it a hypersurface of $X=\P^3\times \P^3$ of bidegree $(1,1)$, denoted by $Y_A$; this is an element of the complete linear system $|\mathcal{O}_X(1,1)|\cong \P^{15}$. The singular locus of $Y_A$, when non-empty, is a product of projective spaces. Hence, the space $\mathcal{X}(3)$ parametrizes hypersurfaces $Y_A$ with a marking: complete collineation of the type $Y_Q$, where $Q\in \mathcal{X}(i)$, for some $i<3$. The hypersurface $Y_A$ is smooth for generic $A\in \mathcal{X}(3)$.

On another hand, let $F(Y_A)\subset \mathbb{G}(1,3)\times \mathbb{G}(1,3)$ be the Fano scheme of  ruled surfaces, isomorphic to $\P^1\times \P^1$, embedded in $Y_A$. This is the zero locus of a global section induced by $Y_A$,

\[
\begin{tikzpicture}[xscale=1.5,yscale=-1.4]
  \node (A0_0) at (0, 0) {$S^{*}\times S^{*}$};
  \node (A1_0) at (0, 1) {$\mathbb{G}(1,3)\times \mathbb{G}(1,3)$};
  \path (A0_0) edge [->]node [auto] {$\scriptstyle{\pi}$} (A1_0);
  \path (A1_0) edge [->,bend right=35]node [auto] {$\scriptstyle{Y_A}$} (A0_0);
\end{tikzpicture}
\]
where $S^*$ is the dual of the tautological bundle over $\mathbb{G}(1,3)$.

We claim that the subscheme $F(Y_A)$ fully determines the hypersurface $Y_A$. Hence, the Hilbert scheme $ \mathbf{Hilb}_{1,1}$ parametrizing subschemes $F(Y_A)$ inside the product $\mathbb{G}(1,3)\times \mathbb{G}(1,3)$ is birational to $\mathcal{X}(3)$. In other words, $\mathbf{Hilb}_{1,1}\cong \mathcal{X}(3)$. We conjecture that this rational map is a flip over the Chow variety that parametrizes the cycle classes of $F(Y_A)$.

\bibliographystyle{amsalpha}
\bibliography{Biblio}
\end{document}